\documentclass[12pt,twoside,letterpaper]{amsart}
\usepackage{amsmath,amssymb,fullpage}   
                            
\usepackage{url}            
\usepackage{graphicx}       
\usepackage{amsthm}

\numberwithin{equation}{section}

\theoremstyle{definition}
\newtheorem{theorem}[equation]{Theorem}

\newtheorem{corollary}[equation]{Corollary}

\newtheorem{remark}[equation]{Remark}

\newcommand{\Id}{\operatorname{Id}}

\newcommand{\Ad}{\operatorname{Ad}}

\newcommand{\HH}{\mathsf{HH}}

\newcommand{\HP}{\mathsf{HP}}
\newcommand{\RR}{\mathbb{R}}

\newcommand{\ZZ}{\mathbb{Z}}
\newcommand{\bbO}{\mathbb{O}}

\newcommand{\iso}{{\;\stackrel{_\sim}{\to}\;}}
\newcommand{\CC}{\mathbb{C}}

\newcommand{\cO}{\mathcal{O}}
\newcommand{\cN}{\mathcal{N}}
\newcommand{\cW}{\mathcal{W}}

\newcommand{\bb}{\mathfrak{b}}

\newcommand{\g}{\mathfrak{g}}
\newcommand{\mm}{\mathfrak{m}}

\newcommand{\gr}{\operatorname{\mathsf{gr}}}
\newcommand{\Aut}{\operatorname{Aut}}
\newcommand{\Stab}{\operatorname{Stab}}

\newcommand{\onto}{\twoheadrightarrow}
\newcommand{\into}{\hookrightarrow}
\newcommand{\Sym}{\operatorname{Sym}}

\newcommand{\Spec}{\operatorname{Spec}}
\newcommand{\Irrep}{\operatorname{Irrep}}

\begin{document}
\title{Traces on finite $\cW$-algebras}
\author{Pavel Etingof and Travis Schedler}
\date{April 22, 2010}
\begin{abstract}
  We compute the space of Poisson traces on a classical $\cW$-algebra,
  i.e., linear functionals invariant under Hamiltonian
  derivations. Modulo any central character, this space identifies
  with the top cohomology of the corresponding Springer fiber. As a
  consequence, we deduce that the zeroth Hochschild homology of the
  corresponding quantum $\cW$-algebra modulo a central character
  identifies with the top cohomology of the corresponding Springer
  fiber. This implies that the number of irreducible
  finite-dimensional representations of this algebra is bounded by the
  dimension of this top cohomology, which was established earlier by
  C. Dodd using reduction to positive characteristic. Finally, we
  prove that the entire cohomology of the Springer fiber identifies
  with the so-called Poisson-de Rham homology (defined previously by
  the authors) of the centrally reduced classical $\cW$-algebra.
\end{abstract}
\maketitle
\section{Introduction}

The main goal of this note is to compute the zeroth Poisson homology
of classical finite $\cW$-algebras, and the zeroth Hochschild homology
of their quantizations.  Modulo any central character, both spaces
turn out to be isomorphic to the top cohomology of the corresponding
Springer fiber. The proof is based on the presentation of the Springer
$D$-module on the nilpotent cone by generators and relations (due to
Hotta and Kashiwara), and earlier results of the authors on the
characterization of zeroth Poisson homology in terms of
$D$-modules. This implies an upper bound on the number of irreducible
finite-dimensional representations of a quantum $\cW$-algebra with a
fixed central character, which was previously established by C. Dodd
using positive characteristic arguments. We also show that the
Poisson-de Rham homology groups of a centrally reduced classical
$\cW$-algebra (defined earlier by the authors) are isomorphic to the
cohomology groups of the Springer fiber in complementary dimension.

\subsection{Definition of classical $\cW$-algebras} 
We first recall the definition of classical $\cW$-algebras (see, e.g.,
\cite{GaGiqss, Losfwa} and the references therein). Let $\g$ be a
finite-dimensional simple Lie algebra over $\CC$ with the
nondegenerate invariant form $\langle -,- \rangle$. We will identify
$\g$ and $\g^*$ using this form.  Let $G$ be the adjoint group
corresponding to $\g$. Fix a nilpotent element $e \in \g$. By
the Jacobson-Morozov theorem, there exists an $\mathfrak{sl}_2$-triple
$(e, h, f)$, i.e., elements of $\g$ satisfying $[e,f] = h, [h,e]=2e$,
and $[h,f]=-2f$.  For $i \in \ZZ$, let $\g_i$ denote the
$h$-eigenspace of $\g$ of eigenvalue $i$.  Equip $\g$ with the
skew-symmetric form $\omega_e(x, y) := \langle e, [x,y]\rangle$.  This
restricts to a symplectic form on $\g_{-1}$. Fix a Lagrangian
$\mathfrak{l} \subset \g_{-1}$, and set
\begin{equation}
\mm_e := \mathfrak{l} \oplus \bigoplus_{i \leq -2} \g_{i}.
\end{equation}
Then, we define a shift of $\mm_e$ by $e$:
\begin{equation}
\mm_{e}' := \{x - \langle e, x \rangle: x \in \mm_e \}\subset \Sym \g.
\end{equation}

The classical $\cW$-algebra $\cW_e$ is
defined to be the Hamiltonian reduction of $\mathfrak{g}$ with respect
to $\mm_e$ and the character $\langle e,\cdot\rangle$, i.e., 
\begin{equation}
\cW_e := (\Sym \g / \mm_e' \cdot \Sym \g)^{\mm_e},
\end{equation}
where the invariants are taken with respect to the adjoint
action. It is well known that, up to isomorphism, $\cW_e$ is
independent of the choice of the $\mathfrak{sl}_2$-triple containing
$e$.

Since it is a Hamiltonian reduction, $\cW_e$ is naturally a Poisson
algebra. The bracket $\{\,, \}: \cW_e \otimes \cW_e \rightarrow \cW_e$
is induced by the standard bracket on $\Sym \g$.  The Poisson center of $\cW_e$
(i.e., elements $z$ such that $\{z, F\}=0$ for all $F$) is isomorphic
to $(\Sym \g)^\g$, by the embedding $(\Sym \g)^\g \into (\Sym
\g)^{\mm_e} \rightarrow (\Sym \g / (\mm_e' \Sym \g))^{\mm_e}$.  It is
known that this composition is injective (since, by Kostant's theorem, the
coset $e+\mm_e$ meets generic semisimple coadjoint orbits of $\g$).

Let $Z := (\Sym \g)^\g$ and $Z_+ = (\g \Sym \g)^\g$ be its
augmentation ideal.  We therefore have an embedding $Z \into \cW_e$,
and can consider the central quotient
\begin{equation}
\cW_e^0 := \cW_e / Z_+ \cW_e.
\end{equation}


\subsection{The Springer correspondence} 

We need to recall a version of the Springer correspondence between
representations of the Weyl group $W$ of $\g$ and certain $G$-equivariant
local systems on nilpotent orbits in $\g$. 

Let $\mathcal{N} \subset \g$ be the nilpotent cone.
  Let $\mathcal{B}$ be the flag variety of $\g$, consisting
of Borel subalgebras $\bb \subset \g$. Consider the
Grothendieck-Springer map $\rho: \widetilde \g := \{(\bb, g): g \in
\bb\} \subset \mathcal{B} \times \g \onto \g$, which restricts to the
Springer resolution $\widetilde{\mathcal{N}}:=\rho^{-1}(\mathcal{N})
\onto \mathcal{N}$.  Note that $\widetilde{\mathcal{N}}\cong
T^*\mathcal{B}$.

Let $W$ be the Weyl group of $\g$ and $\Irrep(W)$ its set of
irreducible representations, up to isomorphism. For each $\chi \in
\Irrep(W)$, denote by $V_\chi$ the underlying vector space and by $\chi:W
\rightarrow \Aut(V_\chi)$ the corresponding representation.  

Then, there is a well known isomorphism (e.g., \cite[Theorem 1.13]{Spr78})
\begin{equation}
H^{\dim_\RR \rho^{-1}(e)}(\rho^{-1}(e)) \cong \bigoplus_{\chi \in \Irrep_e(W)} \psi_\chi \otimes V_\chi,
\end{equation}
where $\Irrep(W) = \bigsqcup_{e \in \mathcal{N}/G} \Irrep_e(W)$, and
for all $\chi \in \Irrep_e(W)$, $\psi_\chi$ is a certain irreducible
representation of the component group $\pi_0(\Stab_G(e))$ of the stabilizer
of $e$ in $G$. For each $\chi \in \Irrep_e(W)$, let us use the
notation $\mathbb{O}_\chi := \bbO(e)=G \cdot e$.

\subsection{The main results}

For any Poisson algebra $A$, we consider the \emph{zeroth Poisson homology}, $\HP_0(A) := A / \{A, A\}$ (which is the same as the zeroth Lie homology). 
Its dual is the space of \emph{Poisson traces}, i.e., linear functionals
$A \rightarrow \CC$ which are invariant under Hamiltonian derivations 
$\lbrace{a, - \rbrace}$. 
 
As a consequence of \cite{GaGiqss} (see also \cite[\S 2.6]{Losfwa}),
there is a natural action of the stabilizer $\Stab_G(e,h,f)$ of the
$\mathfrak{sl}_2$-triple $(e,h,f)$ on $\cW_e$ by Poisson
automorphisms.  This is because of the alternative construction of
$\cW_e$ in \cite{GaGiqss} which is invariant under $\Stab_G(e,h,f)$:
$\cW_e = (\Sym \g / \mathfrak{n}_e' \cdot \Sym \g)^{\mathfrak{p}_e}$,
where $\mathfrak{n}_e' = \{x - \langle e, x \rangle: x \in
\bigoplus_{i \leq -2} \g_i \}$ and $\mathfrak{p}_e = \bigoplus_{i \leq
  -1} \g_i$.

Since this action on $\cW_e$ is Hamiltonian, it gives rise to an
action of $\pi_0(\Stab_G(e,h,f))$ on $\HP_0(\cW_e)$. Note that, since
$\Stab_G(e,h,f)$ is the reductive part of $\Stab_G(e)$, the component
group coincides with $\pi_0(\Stab_G(e))$. 
Clearly, this group also acts on $\HP_0(\cW_e^0)$. 

Our first main result is the following theorem.

\begin{theorem}\label{hp0thm} As $\pi_0(\Stab_G(e))$-representations,
\begin{equation}\label{hp0thmeq}
  \HP_0(\cW_e^0) \cong H^{\dim_\RR(\rho^{-1}(e))}(\rho^{-1}(e)) \cong \bigoplus_{\chi \in \Irrep_e(W)} \psi_\chi \otimes V_\chi.
\end{equation}
\end{theorem}
Here the action on the right hand side is in the first component. 

\begin{remark} 
  There is a slightly different way to view the Springer
  correspondence through \cite{HKihs} which further
  explains the above results. 
Namely, for a smooth variety $X$, denote by
$\Omega_X$ the right $D$-module of volume forms on $X$. 
Then \cite[Theorem 5.3]{HKihs} states that
\begin{equation} \label{th53eqn}
\rho_*(\Omega_{\widetilde{N}}) \cong 
\bigoplus_{\chi \in \Irrep(W)} \mathcal{M}_\chi \otimes V_\chi,
\end{equation}
where $\mathcal{M}_\chi$ are irreducible, 
holonomic, pairwise nonisomorphic $G$-equivariant right $D$-modules 
on $\mathcal{N}$.

Each $D$-module $\mathcal{M}_\chi$ is uniquely determined by
its support, which is the closure, $\overline{\bbO}$, of a nilpotent
coadjoint orbit $\bbO = \bbO(e) \subset \mathcal{N}$ (i.e., a
symplectic leaf of $\mathcal{N}$), together with a $G$-equivariant
local system on $\bbO$ (the restriction of ${\mathcal M}_\chi$ to
$\Bbb O$).  Then, $\bbO = \bbO_\chi$, and
the local system is $\psi_\chi$.

Taking the pushforward of \eqref{th53eqn} to a point, one can deduce
that $\HP_0(\cW_e^0)$ is isomorphic to the RHS of \eqref{hp0thmeq} using
the method of \cite{ESdm} recalled in \S \ref{esdms} below.
\end{remark}

Next, let $U\g$ denote the universal enveloping algebra of $\g$, and
let $\cW_e^{q} := (U\g / \mm_e' U \g)^{\mm_e}$ be the quantum
$\cW$-algebra, which is a filtered (in general, noncommutative)
algebra whose associated graded algebra is $\cW_e$, as a Poisson
algebra.  The center of $\cW_e^q$ is an isomorphic image of $Z(U\g)$,
which is identified with $Z$ as an algebra via the Harish-Chandra
isomorphism.  Let $\eta: Z \rightarrow \CC$ be a character, and define
the algebras $\cW_e^{\eta} := \cW_e/(\ker(\eta))$ and $\cW_e^{q,\eta}
:= \cW_e^{q}/(\ker(\eta))$. These are filtered Poisson (respectively,
associative) algebras whose associated graded algebras are $\cW_e^0$.
Moreover, using the construction of \cite{GaGiqss} as above (i.e.,
$\cW_e^q \cong (U\g / \mathfrak{n}_e' U\g)^{\mathfrak{p}_e}$),
$\cW_e^{q}$ as well as $\cW_e^{q,\eta}$ admit actions of
$\Stab_G(e,h,f)$ (as does $\cW_e^{\eta}$, for all $\eta$).
Since this action is Hamiltonian, $\HP_0(\cW_e^{\eta})$ and
$\HH_0(\cW_e^{q,\eta})$ admit actions of $\pi_0(\Stab_G(e)) =
\pi_0(\Stab_G(e,h,f))$ for all $\eta$.

Consider the zeroth Hochschild homology 
$\HH_0(\cW_e^{q,\eta}) :=
\cW_e^{q,\eta}/[\cW_e^{q,\eta}, \cW_e^{q,\eta}]$.
There is a canonical surjection $\HP_0(\cW_e^0) \onto \gr \HH_0(\cW_e^{q,\eta})$.  

\begin{theorem} \label{hp0ccthm}
(i) The canonical surjection $\HP_0(\cW_e^0) \iso \gr \HH_0(\cW_e^{q,\eta})$ is an isomorphism.

(ii)  The families $\HP_0(\cW_e^{\eta})$ and $\HH_0(\cW_e^{q,\eta})$ are
  flat in $\eta$.  In particular, for all $\eta$, they are isomorphic
  to the top cohomology of the Springer fiber, $H^{\dim_\RR
    \rho^{-1}(e)}(\rho^{-1}(e))$, as representations of $\pi_0(\Stab_G(e))$.

  (iii) The groups $\HP_0(\cW_e)$ and $\HH_0(\cW_e^q)$ are
isomorphic to $Z\otimes H^{\dim_\RR \rho^{-1}(e)}(\rho^{-1}(e))$ 
as $Z[\pi_0(\Stab_G(e))]$-modules.
\end{theorem}

Theorem \ref{hp0ccthm} follows from Theorem 
\ref{hpdrthm}, as explained below.

\begin{corollary}\label{dodd} (C. Dodd, \cite{D10})
For every central character $\eta$, 
the number of distinct 
irreducible finite-dimensional representations of
$\cW_e^{q,\eta}$ is at most $\dim H^{\dim_{\Bbb
R}\rho^{-1}(e)}(\rho^{-1}(e))$.  
\end{corollary}

\begin{proof}
  This immediately follows from the above theorem, because the number
  of isomorphism classes of irreducible finite-dimensional
  representations of any associative algebra $A$ is dominated by $\dim
  A/[A,A]$ (since characters of nonisomorphic irreducible
  representations are linearly independent functionals on $A/[A,A]$).
\end{proof}

\begin{remark}
The argument in the appendix to \cite{ESdm} 
by I. Losev together with Corollary \ref{dodd} also implies an 
upper bound on the number $N_e$ of prime (or, equivalently, primitive)
ideals in $\cW_e^{q,\eta}$. For every nilpotent orbit $\mathbb{O}_{e'}$ whose
closure contains $e$, let $M_{e,e'}$ denote the number of irreducible components  of the intersection $\overline{\mathbb{O}_{e'}} \cap \Spec \cW_e$ of the closure of the orbit
$\mathbb{O}_{e'}$ with the Kostant-Slodowy slice to $e$.  Then,
$$
N_e\le \sum_{\overline{\mathbb{O}_{e'}} \ni e} M_{e,e'} \cdot 
\dim H^{\dim_\RR\rho^{-1}(e')}(\rho^{-1}(e')),
$$
where the sum is over the distinct orbits $\mathbb{O}_{e'}$ whose
closure contains $e$.  Briefly, we explain this as follows: Losev's
appendix to \cite{ESdm} gives a map from finite-dimensional
irreducible representations of $\cW_{e'}^{q,\eta}$ to prime ideals of
$\cW_e^{q,\eta}$ supported on the irreducible component of
$\mathbb{O}_{e'} \cap \Spec \cW_e$ containing $e'$, and shows that all
prime ideals are constructed in this way. (More precisely, in
\emph{op.~cit.}, a construction is given of all prime ideals of
filtered quantizations of affine Poisson varieties with finitely many
symplectic leaves, which specializes to this one since the
aforementioned irreducible components coincide with the symplectic
leaves of $\Spec \cW_e^0$, and $\cW_{e}^{q,\eta}$ and
$\cW_{e'}^{q,\eta}$ are quantizations of $\cW_e^0$ and $\cW_{e'}^0$,
respectively.)  Then, the bound follows from Corollary \ref{dodd}.
\end{remark}

\subsection{Higher homology}
Finally, following \cite{ESdm}, one may consider the higher
\emph{Poisson-de Rham homology} groups, $\HP^{DR}_i(\cW_e^\eta)$, of
$\cW_e^\eta$, whose definition we recall in the following section.
Here, we only need to know that $\HP^{DR}_0(A)=\HP_0(A)$ for all
Poisson algebras $A$ (although the same is \emph{not} true for higher
homology groups). Let $\eta: Z \rightarrow \CC$ be an \emph{arbitrary}
central character.  

The following theorem is a direct generalization of Theorem
\ref{hp0thm}. Hence, we only
prove this theorem, and omit the proof of the aforementioned
theorem.
\begin{theorem} \label{hpdrthm} 
As $\pi_0(\Stab_G(e))$-representations,
$\HP^{DR}_i(\cW_e^\eta) \cong H^{\dim_\RR \rho^{-1}(e) - i}(\rho^{-1}(e))$. Moreover, for
generic $\eta$, $\HH_i(\cW_e^{q,\eta})$ is also isomorphic to these.
\end{theorem}
\begin{remark}
For $e=0$ one has $\cW_e^q=U\g$, and the algebras 
$\cW_e^{q,\eta}$ are the maximal primitive quotients of $U\g$.
In this case, the last statement of Theorem \ref{hpdrthm} holds for all regular 
characters $\eta$ (see \cite{Soehoc} and \cite{VdBdual, VdBdualerr}).
On the other hand, the genericity assumption for $\eta$ cannot be 
removed for $i>0$. Namely, for non-regular values of $\eta$, it is, in general, 
\emph{not} true that $\HH_i(\cW_e^{q,\eta})$ is isomorphic to the cohomology $H^{\dim_\RR
 \rho^{-1}(e) - i}(\rho^{-1}(e))$ of the Springer fiber.  For
example, when $e=0$ in $\g = \mathfrak{sl}_2$,
  then the variety $\Spec \cW_e^0$ is the cone
  $\CC^2/\ZZ_2$. In this case, by \cite[Theorem
  2.1]{FSSAhhc} (and the preceding comments), $\HH_i(\cW_e^{q,\eta})
  \neq 0$ for all $i \geq 3$ when $\eta: Z(U \g) \rightarrow \CC$ is
  the special central character corresponding to the Verma module with
  highest weight $-1$, i.e., the character corresponding to the fixed
  point of the Cartan $\mathfrak{h}$ under the shifted Weyl group
  action.
\end{remark}

\section{The construction of \cite{ESdm}}\label{esdms}
We prove Theorem \ref{hpdrthm} using the method
of \cite{ESdm}, which we now recall.  

To a smooth affine Poisson variety $X$, we attached the right
$D$-module $M_X$ on $X$ defined as the quotient of the algebra of
differential operators $D_X$ by the right ideal generated by
Hamiltonian vector fields.  Then, $\HP_0(\cO_X)$ identifies with the
(underived) pushforward $M_X \otimes_{D_X} \cO_X$ of $M_X$ to a point.

More generally, if $X$ is not necessarily smooth, but equipped with a
closed embedding $i: X \into V$ into a smooth affine variety $V$
(which need not be Poisson), we defined the right $D$-module $M_{X,i}$
on $V$ as the quotient of $D_V$ by the right ideal generated by
functions on $V$ vanishing on $X$ and vector fields on $V$ tangent to $X$ 
which restrict on $X$ to Hamiltonian vector fields.  This is independent of the choice of
embedding, in the sense that the resulting $D$-modules on $V$
supported on $X$ correspond to the same $D$-module on $X$ (up to a
canonical isomorphism) via Kashiwara's theorem. Call this $D$-module
$M_X$.  The pushforward of $M_{X}$ to a point remains isomorphic to
$\HP_0(\cO_X)$.

More generally, for an arbitrary affine variety $X$, we defined the
groups $\HP^{DR}_i(\cO_X)$ as the full (left derived) pushforward of
$M_X$ to a point.

\section{Proof of Theorem \ref{hpdrthm}}
  
Our main tool is 
\begin{theorem}\cite[Theorem 4.2 and Proposition 4.8.1.(2)]{HKihs} (see also \cite[\S
  7]{LSssi}) \label{hkthm}
\begin{equation}
M_{\mathcal N} \cong \rho_*(\Omega_{\widetilde {\mathcal N}}).
\end{equation}
\end{theorem}

We now begin the proof of Theorem \ref{hpdrthm}.  First, take $\eta =
0$. Since $M_{\mathcal{N}} \cong
\rho_*(\Omega_{\widetilde{\mathcal{N}}})$, it follows that, letting $\pi$ and $\widetilde \pi$ 
denote the projections of ${\mathcal N}$ and $\widetilde{\mathcal N}$ 
to a point, 
$$
\HP^{DR}_i(\cO_{\mathcal{N}}) = L_i \pi_*(M_\mathcal{N})\cong
L_i {\widetilde \pi}_*(\Omega_{\widetilde{\mathcal{N}}}) \cong H^{\dim_\CC
  \widetilde{\mathcal{N}} -i}(\widetilde{\mathcal{N}}).
$$
Similarly, if we consider $\Spec \cW_e^0 \subseteq \mathcal{N}$ (the
intersection of a Kostant-Slodowy slice to the orbit of $e$ with
$\mathcal{N}$), then it follows, viewing all varieties as embedded in
the smooth variety $\g$, that $\rho_*(\Omega_{\rho^{-1}(\Spec
  \cW_e^0)}) \cong M_{\Spec \cW_e^0}$, since $\rho^{-1}(\Spec
\cW_e^0) \subset \widetilde{\cN}$ is smooth.  In more detail, let
$Y_e$ denote a formal neighborhood of $\Spec \cW_e^0$ in $\cN$,
$\widetilde Y_e$ denote a formal neighborhood of $\rho^{-1}(\Spec
\cW_e^0)$ in $\widetilde \cN$, and $\widehat{[e,\g]}$ denote a formal
completion of $[e,\g]$ at $0$. Then, $\widetilde Y_e \cong \rho^{-1}(\Spec
\cW_e^0) \times \widehat{[e, \g]}$ and $Y_e \cong \Spec \cW_e^0 \times
\widehat{[e,\g]}$. With these identifications, $\rho|_{\widetilde Y_e}
= \rho|_{\rho^{-1}(\Spec \cW_e^0)} \times \Id_{\widehat{[e, \g]}}$.
Then, $\Omega_{\widetilde Y_e} \cong \Omega_{(\rho^{-1}(\Spec
  \cW_e^0)} \boxtimes \Omega_{\widehat{[e, \g]}}$ and $M_{Y_e} \cong
\Omega_{\Spec \cW_e^0} \boxtimes \Omega_{\widehat{[e, \g]}}$.  Since
$M_{\cN} \cong \rho_* (\Omega_{\widetilde \cN})$, restricting to
$\widetilde Y_e$ yields $M_{Y_e} \cong \rho_*( \Omega_{\widetilde
  Y_e})$, and we conclude from the above that $M_{\Spec \cW_e^0} \cong
\rho_* (\Omega_{\rho^{-1}(\Spec \cW_e^0)})$.

 Since $\rho^{-1}(\Spec \cW_e^0) \rightarrow \Spec \cW_e^0$ is birational,
$\dim_\CC \rho^{-1}(\Spec \cW_e^0) = \dim_\CC \Spec
\cW_e^0$.  Hence, $\HP^{DR}_i(\cW_e^0) \cong H^{\dim_\CC \Spec \cW_e^0 -
  i}(\rho^{-1}(\Spec \cW_e^0))$.  Next, observe that the contracting
$\CC^*$-action on $\Spec \cW_e^0$ lifts to a deformation retraction of
$\rho^{-1}(\Spec \cW_e^0)$ to $\rho^{-1}(e)$, as topological spaces
(in the complex topology).  Moreover, $\rho^{-1}(e)$ is compact,
and hence its dimension must equal the degree of the top cohomology,
$\dim_\RR \rho^{-1}(e) = \dim_\CC \rho^{-1}(\Spec \cW_e^0) = \dim_\CC
\Spec \cW_e^0$.  (This can also be computed directly: all of these
quantities are equal to the complex codimension of $G \cdot e$
inside $\mathcal{N}$.)  We conclude the first equality of the theorem
for $\eta = 0$, i.e., $\HP^{DR}_i(\cW_e^0) \cong H^{\dim_R
  \rho^{-1}(e)}(\rho^{-1}(e))$.

Since the parameter space of $\eta$ has a contracting $\CC^*$ action
with fixed point $\eta = 0$, to prove flatness of
$\HP^{DR}_i(\cW_e^\eta)$, it suffices to show that $\dim
\HP^{DR}_i(\cW_e^\eta) = \dim \HP^{DR}_i(\cW_e^0)$ for generic $\eta$.
For generic $\eta$, $\Spec \cW_e^\eta$ is smooth and symplectic, and
hence (by \cite[Example 2.2]{ESdm}), $M_{\Spec \cW_e^\eta} =
\Omega_{\Spec \cW_e^\eta}$, so that $\HP^{DR}_i(\cW_e^\eta) \cong
H^{\dim \Spec \cW_e^\eta - i}(\Spec W_e^\eta)$.  Moreover,
$\displaystyle \rho^{-1}(\Spec \cW_e^\eta) \mathop{\iso}^{\rho} \Spec
\cW_e^\eta$.  Next, for all $\eta$, the family $\rho^{-1}(\Spec
\cW_e^\eta)$ is topologically trivial \cite{Slflsgs} (see also
\cite{Slsssag}), and
hence its cohomology has constant dimension, and equals $\dim
H^{\dim_\CC \Spec \cW_e^0 - i}(\rho^{-1}(\Spec \cW_e^0))$.  Hence, for
generic $\eta$, $\dim \HP^{DR}_i(\cW_e^\eta) = \dim
\HP^{DR}_i(\cW_e^0)$, as desired.

Let us now prove the second statement of the theorem.  Let $\hbar$ be
a formal parameter. For any central character $\eta: Z \rightarrow \CC$, consider
the character $\eta/\hbar: Z((\hbar)) \rightarrow \CC((\hbar))$.  Let $\cW_e^{q,\eta/\hbar} :=  \cW_e^{q}((\hbar)) / \ker(\eta/\hbar)$.  As we will explain below, by results of Nest-Tsygan and Brylinski, for generic $\eta$, $\HH_i(\cW_e^{q, \eta/\hbar}) = \HP_i(\cW_e^{\eta})((\hbar)) =
H^{\dim_\CC \cW_e^\eta-i}(\Spec \cW_e^\eta,\Bbb C((\hbar)))$.  Hence,
$\HH_i(\cW_e^{q,\eta/\hbar}) \cong \HP_i^{DR}(\cW_e^0)((\hbar))$ for generic
$\eta$.  This implies the statement.\footnote{A similar argument is
  used in the proof of \cite[Theorem 1.8.(ii)]{EGsra}.}

In more detail, $\cW_e^{q,\eta/\hbar}$ is obtained from a deformation
quantization of $\cW_e^\eta$ in the following way.  Let
$\cW_e^{\hbar}$ be the $\hbar$-adically completed Rees algebra $\widehat{\bigoplus}_{m \geq 0} \hbar^m F^m
\cW_e^q$, where $F^\bullet \cW_e^q$ is the filtration on
$\cW_e^q$. This is a deformation quantization of $\cW_e$.  Consider
the quotient $\cW_e^{\hbar,\eta} := \cW_e^\hbar / (\ker(\eta))$.
Then, $\cW_e^{\hbar, \eta}$ is a deformation quantization of
$\cW_e^{\eta}$.  (Recall that, in general, a deformation quantization
of a Poisson algebra $A_0$ is an algebra of the form $A_\hbar =
(A_0[[\hbar]],\star)$, $A_0[[\hbar]] := \{\sum_{i \geq 0} a_i \hbar^i,
a_i \in A_0\}$ satisfying $a \star b = ab + O(\hbar)$ and $a \star b -
b \star a = \hbar \{a,b\} + O(\hbar^2)$, up to an isomorphism.) 
Then, by \cite[Theorems 2.2.1 and 3.1.1]{Br} and
\cite[Theorem A2.1]{NeTs}, if $\cW_e^\eta$ is smooth and symplectic,
then $\HH_i(\cW_e^{\hbar,\eta}[\hbar^{-1}]) \cong
\HP_i(\cW_e^{\eta})((\hbar)) \cong H^{\dim_\CC \cW_e^\eta-i}(\Spec
\cW_e^\eta,\Bbb C((\hbar)))$.  Furthermore, the map defined by $x
\mapsto \hbar x$, $x\in {\mathfrak g}$ defines an isomorphism
$\cW_e^{\hbar,\eta}[\hbar^{-1}] \iso \cW_e^{q,\eta/\hbar}$. Since
$\cW_e^\eta$ is smooth for generic $\eta$, this implies the results
claimed in the previous paragraph.

\section{Proof of Theorem \ref{hp0ccthm}}\label{hp0ccthmpfs}
Note that (iii) easily follows from (ii), since
$\HP_0(\cW_e)$ is a finitely generated $\Bbb Z_+$-graded 
$Z$-module, which is flat (i.e., projective) by Theorem
\ref{hpdrthm}.

Thus, it suffices to prove that $\dim \HH_0(\cW_e^{q,\eta}) = \dim
H^{\dim_\RR \rho^{-1}(e)}(\rho^{-1}(e))$ for all central characters
$\eta$.  As remarked before the statement of Theorem \ref{hp0ccthm},
there is a canonical surjection $\HP_0(\cW_e^0) \onto \gr
\HH_0(\cW_e^{q,\eta})$. Hence, for all $\eta$, $\dim
\HH_0(\cW_e^{q,\eta}) \leq \dim \HP_0(\cW_e^0)$, which equals $\dim
H^{\dim_\RR \rho^{-1}(e)}(\rho^{-1}(e))$ by Theorem \ref{hp0thm}
(which follows from Theorem \ref{hpdrthm}). The minimum value of $\dim
\HH_0(\cW_e^{q,\eta})$ is attained for generic $\eta$ (since
$\HH_0(\cW_e^q)$ is a finitely generated $Z$-module), where it is also
$\dim H^{\dim_\RR \rho^{-1}(e)}(\rho^{-1}(e))$ by Theorem
\ref{hpdrthm}. Hence, this dimension must be the same for all $\eta$.

\textbf{Acknowledgments.}  We are grateful to Roman Bezrukavnikov and
Ivan Losev for useful discussions. The first author's work was
partially supported by the NSF grant DMS-0504847. The second author is
a five-year fellow of the American Institute of Mathematics, and was
partially supported by the ARRA-funded NSF grant DMS-0900233.

\bibliographystyle{amsalpha}
\bibliography{master}
\end{document}